\documentclass[12pt]{amsart}

\newtheorem{thm}{Theorem}
\newtheorem{prop}{Proposition}[section]
\newtheorem{lem}[prop]{Lemma}

\theoremstyle{definition}

\newtheorem{notations}[prop]{Notations}
\numberwithin{equation}{section}
\newcommand\N{\mathbb{N}}
\newcommand\Z{\mathbb{Z}}
\newcommand\R{\mathbb{R}}
\newcommand\C{\mathbb{C}}

\newcommand\Log{{\mathrm Log}}
\newcommand\e{\varepsilon}
\newcommand\undemi{\frac{1}{2}}
\newcommand\Ad{A_\delta}

\newcommand\un{u_{ \e_n,\delta_n}}

\begin{document}

\title[Periodic unfolding and the G.-L. equation] {Periodic unfolding and Homogenization \\  for the Ginzburg-Landau Equation \\
{\rm Preliminary draft}
}%
\author{Myrto Sauvageot}%
\address{}%
\email{sauvageo@ann.jussieu.fr}%

\thanks{}

\date{March 30, 2009}%

\begin{abstract} We investigate, on a bounded domain $\Omega$ of $\R^2$ with fixed $S^1$-valued boundary condition $g$ of degree $d>0$, the asymptotic behaviour of solutions $u_{\varepsilon,\delta}$ of a class of Ginzburg-Landau equations driven by two parameter\,: the usual Ginzburg-Landau parameter, denoted $\varepsilon$, and the scale parameter $\delta$ of a geometry provided by a field of $2\times 2$ positive definite matrices $x\to A(\frac{x}{\delta})$. The field $\R^2\ni x\to A(x)$ is of class $W^{2,\infty}$ and periodic. We show, for a suitable choice of the $\varepsilon$'s depending on $\delta$, the existence of a limit configuration $u_\infty\in H^1_g(\Omega,S^1)$, which, out of a finite set of singular points,  is a weak solution of the equation of $S^1$-valued harmonic functions for the geometry related to the usual homogenized matrix $A^0$.

\end{abstract}

\maketitle


\section{Introduction and statement of the  results.}

\subsection{\label{11} Periodic homogenization}\-

\smallskip Homogenization can be considered as the mathematical theory of the macroscopic behaviour  of composite material.

In the case of periodic homogenization, one supposes that heterogeneities 	are distributed with a periodicity of length $\delta$, small with respect of the size of the medium, and that the geometry of their distributions is described by a field $x\to \displaystyle A\big(\frac{x}{\delta}\big)$     of  $n\times n$ positive definite matrices with real entries. In this setting, the unscaled field $\R^n \ni y\to A(y)$ is periodic, with a period associated with a cell $Y\subset \R^n$. Usually, one studies the simpler case $Y=[0,1[^n$, which corresponds to a period equal to $1$ in the $n$ directions of the space. The purpose is to study the asymptotic geometry as the scaling parameter $\delta$ tends to $0$, which provides the properties of the system for infinitely small $\delta$.

\smallskip Consider for instance the paradigmatic problem for periodic homogenization\,: describe the limit behaviour of a family $U^\delta$ of variational solutions for the following system of equations\,:
\begin{equation}
-div\Big( A\big(\frac{x}{\delta}\big)\nabla U^\delta(x)\Big)= f \text{ on }\Omega\,,\; U^\delta=0 \text{ on }\partial \Omega.
\end{equation}
It is now a classical result that, under suitable assumptions on the sectrum of the matrices $A(y)$, a weak $H^1$-limit $U^0$ of the family $U^\delta$ will be a solution of the homogenized system
\begin{equation}
-div\Big( A^0\nabla U^\delta(x)\Big)= f \text{ on }\Omega\,,\; U^\delta=0 \text{ on }\partial \Omega\,,
\end{equation}
where $A^0$, the so-called {\it homogenized matrix of the field $A(\cdot)$}, is an elliptic matrix with constant coefficients (i.e. describing an homogeneous material), which can be explicitely described from the field $A(\cdot)$. One can refer for instance to [SP], [Ba] or [BLP].

The construction of the matrix $A^0$ is recalled in subsection \ref{71}.

\medskip
\subsection{The periodic unfolding method.}\-

\medskip  The two-scale convergence method introduced by Nguetseng [ Ng], and developed by G. Allaire ([Al1 ], [ Al2],) allows to solve more general homogenization problem. Recently,  D. cioranescu, A. Damlamian and G. Griso  have developed a rather quick way to obtain two-scale convergence results, namely the {\it periodic unfolding method}. It is based on a simply defined unfolding operator $T_\delta$, depending on the scaling parameter $\delta$, which transforms a function $f(x)$ on a domain $\Omega$ into a function $T_\delta f(x,y)$ on the cartesian product $\Omega\times Y$. Up to some attention to be paid when one gets close to the boundary of $\Omega$, it  is an isometric operator for any $L^p$-norm and it behaves rather well with respect to differential operators. Periodic unfolding is thoroughly explained in the survey paper [CDG2].

\smallskip Since periodic unfolding first appeared in [CDG1 ], the method has been applied to many linear and nonlinear situations (cf. for instance  [ ], [ ], ...)  Our purpose in this work is to apply the method to the homogenization of the Ginzburg-Landau equation of [BBH], which is some kind of archetypal nonlinear equation and can be considered as a simplified approach to the Ginzburg-Landau model of superconductivity.

\medskip
\subsection{ Homogenizing the Ginzburg-Landau equation.}\-

The problem raised by [BBH] is the asymptotic behaviour, as $\e\to 0$, of the minimizers $u_\e$ of a Ginzburg-Landau energy functional
\begin{equation}
E_\e(u)=\frac{1}{2}\int_\Omega |\nabla u|^2 +\frac{1}{4\e^2}\int_\Omega \big(1-|u|^2\big)^2\,,\; u\in H^1_g(\Omega,\C)\,,
\end{equation}
with $\Omega$ a bounded subset of $\R^2$, ang $g\,:\,\partial \Omega \to S^1$ a modulus $1$ fixed boundary condition of degree $d$. We suppose $d>0$.

Roughly speaking, they prove the existence of a sequence $\e_n$ and of a finite subset $\{a_1,\cdots,a_n\}$ of $\Omega$, such that the locally $H^1$-limit
 $u_*=\lim_n u_{\e_n}$ exists in $H^1_{loc}(\Omega,\C)$. $u_*$ is a modulus $1$ function, and a solution of the equation of $S^1$-valued harmonic functions
 \begin{equation}\label{harm}
-div\,\nabla u_* =u_*\,|\nabla u_*|^2\,.
\end{equation}

\smallskip The same result can be obtained, replacing the functional $E_\e$ by the energy functional $E_\e$ by the functional
\begin{equation}
\frac{1}{2}\int_\Omega\,a(x)\, |\nabla u|^2 +\frac{1}{4\e^2}\int_\Omega \big(1-|u|^2\big)^2\,,
\end{equation}
with $a(x)$ taking values in $\R^*_+$ bounded above by $M$ and below by $m$, $0<m<M$. (For this result, cf. [Be]).

As a byproduct of the present work, i.e. adapting to a simpler context the methods of the present paper, one can show that the result is still valid for minimizers of a energy functional of the form
\begin{equation}\label{GLA}
\frac{1}{2}\int_\Omega\,\nabla u(x)\,A(x)\,\nabla u(x) +\frac{1}{4\e^2}\int_\Omega \big(1-|u|^2\big)^2\,,
\end{equation}
with $\Omega\ni x \to A(x)$ a field of positive definite $2\times 2$ matrices, with $spectrum\,A(x)\subset [m,M]$ as above.

\medskip Related to homogenization is the work of L. Berlyand and P. Mironescu [BM] on the classical Ginzburg-Landau energy (\ref{}) for perforated domains with periodic holes of diameter $\delta$.

\medskip In this paper, we shall consider the homogenized version of (\ref{GLA}), i.e. the perturbated Ginzburg-Landau functional
\begin{equation}
\frac{1}{2}\int_\Omega\,\nabla u(x)\,A\big(\frac{x}{\delta}\big)\,\nabla u(x) +\frac{1}{4\e^2}\int_\Omega \big(1-|u|^2\big)^2\,,
\end{equation}
where $\delta$ is an homogenization parameter. We prove that, adjusting the parameter $\e$ according to $\delta$, a similar result can be obtained, where the notion of $S^1$-valued harmonic function, as in (\ref{harm}), will refer to the geometry provided  by the homogenized matrix $A^0$\,.

\bigskip

\subsection{\label{data}Data.}\-

\bigskip
Throughout this work, we  consider the following data\,:
\par\noindent $\quad {\mathbf \cdot}$ A bounded, connected  open domain $\Omega$ in $R^2$, with class $C^1$-boundary $\partial \Omega$\,;
\par\noindent $\quad {\mathbf \cdot}$ A modulus $1$ boundary condition $g\in C^1(\partial\Omega), S^1)$\,, of degree $d>0$\, (note that the easier case $d=0$ is solved in [ Me]\,)\,;
\par\noindent $\quad {\mathbf \cdot}$ A field $\R^2\ni x\to A(x)\in M^2(\R)$ of $2\times 2$ symmetric  definite positive square matrices which is of class $W^{2,\infty}$, with $x\to A(x)^{-1}$ also bounded, and $\Z^2$-periodic. In other words\,:
\begin{itemize}
\item[] $\exists\,0<m<M$\,, $spectrum(A(x))\subset [m,M]$, $\forall\,x\in\R^2$, i.e.
\begin{equation}\label{mM}
m ||\xi||^2 \leq \xi \cdot A(x)\xi\leq M||\xi||^2\,,\;\forall \xi,x\in\R^2\,;
\end{equation}
\item[] $A(x+z)=A(x)$\,, $\forall x\in\R^2\,,\;z\in\Z^2$\,;
\item[] $A(\cdot)\in W^{2,\infty}(\R^2,M_2(\R))$\,;
\end{itemize}

\par\noindent $\quad {\mathbf \cdot}$ For any $\delta>0$, the field of matrices $\displaystyle \Ad\,:\;\Ad(x)=A\big(\frac{x}{\delta}\big)$, $x\in\R^2$\,;
\par\noindent $\quad {\mathbf \cdot}$ For any $\e,\delta>0$, a minimizer $u_{\e,\delta}$ of the perturbated Ginzburg-Landau energy $E_{\e,\delta}$ on $\Omega$, with boundary condition $g$\,:
\begin{equation}\label{defE}
E_{\e,\delta}(u)=\undemi\int_\Omega \nabla u\cdot A_\delta \nabla u +\frac{1}{4\e^2}\int_\Omega \big(1-|u|^2\big)^2\,,\;u\in H^1_g(\Omega,\C)\,.
\end{equation}

\smallskip Note that $u_{\e,\delta}$ is a solution of the Euler-Lagrange equation of $E_{\e,\delta}$\,:
\begin{equation}\label{Euler}
-div\big(A_\delta\,\nabla u_{\e,\delta}\big)=\frac{1}{\e^2}\,u_{\e,\delta}\big(1-|u_{\e,\delta}|^2\big)\,.
\end{equation}

\subsection{ Statement of the results.}\-

Our final result can be summarized as a theorem\,:

\begin{thm}\label{thm1} Let $(\delta_n)$ be  a sequence in $\R_+^*$ decreasing to $0$. Then, replacing it by a subsequence (stiil denoted $(\delta_n)$\,), one can find a sequence $(\e_n)$ decreasing to $0$ in $\R_+^*$, and a finite subset $\{a_1,\cdots,a_N\}$ of $\Omega$ such that
\,:

1/ $u_\infty=\lim_n u_{\e_n,\delta_n}$ exists in $H^1_{loc}(\Omega\,\backslash\{a_1,\cdots,a_N\},\C)$\,, as a weak locally $H^1$-limit.

2/ $u_\infty$ takes values in $S^1$, coincides with $g$ on the boundary $\partial \Omega$, and is a weak solution, in $\Omega\,\backslash\{a_1,\cdots,a_d\}$, of the equation of $S^1$-valued $A^0$-harmonic functions
\begin{equation}\label{A0harm}
-div(A^0\nabla u_\infty)=u_\infty\,\big(\nabla u_\infty \cdot A^0\nabla u_\infty\big)\,,
\end{equation}
where $A^0$ is the homogenized $2\times 2$-matrix (positive definite, with constant coefficients) associated with the field $A(.)$ {\rm (Cf. subsections \ref{11} and \ref{71}).}
\end{thm}

\smallskip\noindent [{\bf n.b.} As a matter of fact, one can show that there are exactly $d$ singular points $a_1,\cdots,a_d$, but proving this  requires a much longer and much more complicated proof. We shall skip this point, in order to provide for our Theorem 1. a proof which, in some regards, can be considered as surprisingly simple.

\bigskip

\subsection{Organization of the paper.}\-

\medskip The proof is divided in three steps. The first step consists in identifying the set $\{a_1,\cdots,a_N\}$ of the singular points. This is based on
  the classical approach of [BBH], and also some improvements of [S=Struwe] or [B=Beaulieu] \,, together with the necessary adaptation of those previous methods to our context of a highly perturbated geometry, varying with the parameter $\delta$.

  \smallskip
  The second step is the existence of  a limit configuration. The final result is that the sequences $u_{\e_n,\delta_n}$ constructed in step 1 are bounded in $H^1_{loc}(\Omega\backslash \{a_1,\cdots,a_N\}, \C)$, so that they have weak locally $H^1$ limits $u_\infty$. This is done by a quick study of the energy of some  classes of $H^1$ maps, defined on an annulus $\{x\,/\, \lambda \e_n \leq |x-a_i|\leq R \}$, and taking values in $S^1$. Comparing such energies with the energy of $u_{\e_n,\delta_n}$ on the same annulus will lead to the result.

  The point here is that very little knowledge on the $S^1$-valued maps on such an annulus is actually needed. In particular, no explicit computation of their energy is required. One identifies a specific class of $S^1$-valued maps on the annulus which are of given degree, and then shows the existence of a distinguished representative in this class, with two properties\,: first, its value at the boundary of the annulus is prescribed, and then, its energy differs from the minimal energy in the class by only a bounded quantity, with explicit bound. This result suffices to prove, later, that for any $R>0$, the quantity $\displaystyle \sup_n \int_{\Omega\backslash \big( \cup_i B(a_i,R)\big)
  } \big|\nabla u_{\e_n,\delta_n}\big|^2$ is finite.

\smallskip The third step is the proof of the equation satisfied by a limit configuration $u_\infty$, i.e. Equation (\ref{A0harm}) of Theorem \ref{thm1}. Invoking the periodicity of the field of matrices $A(.)$ (which so far had not been taken in account), one can introduce the now well established method of {\it periodic unfolding} of [CDG1]. This method provides a shortcut to the matrix $A^0$, since this matrix appears naturally here as a byproduct of computations which are quite natural, based on the behaviour of periodic unfolding under weak $H^1$-limits, as established in [CDG2].

\bigskip
\section{First properties of minimizers}

\bigskip
\begin{lem}\label{uborne} $|u_{\e,\delta}|\leq 1$, $\forall \,\e,\delta$\,.
\end{lem}
\begin{proof} Check that $\int_\Omega \nabla(f(u))\cdot A_\delta \nabla(f(u)) \leq \int_\Omega \nabla u\cdot A_\delta \nabla u$, where $f(z)=z$ if $|z|\leq 1$, $f(z)=z/|z|$ if $|z|\geq 1$.
\end{proof}

\begin{lem}\label{bornes} There is a constant $C$ such that
\begin{equation}\label{majmu}
m \pi d \log(\frac{1}{\e})-C\leq
E_{\e,\delta}(u_{\e,\delta})\leq M\pi d \log(\frac{1}{\e})+C\,.
\end{equation}
\end{lem}
\begin{proof} Let $E_\e:$ $E_\e(u)=\undemi\int_\Omega |\nabla u|^2 +\frac{1}{4\e^2}\int_\Omega \big(1-|u|^2\big)^2$ be the usual G.-L. energy, as considered in [BBH], and $u_\e$ a minimizer for $E_\e$ in $H^1_g(\Omega,\C)$\,. By estimates in [BBH], one has
\begin{equation}
-\pi d\, \Log(\e)-C' \leq E_\e(u_\e)\leq -\pi d\, \Log(\e)+C'
\end{equation}
for some constant $C'$ depending only on $\Omega$ and $g$.

One will have then, by (\ref{mM})
\begin{equation*}\begin{split}
E_{\e,\delta}(u_{\e,\delta}) & \geq \frac{m}{2}\int_\Omega |\nabla u_{\e,\delta}|^2+\frac{1}{\e^2}\big(1-|u_{\e,\delta}|^2 \big)^2 \\
&=m\,E_{\sqrt m\,\e}(u_{\e,\delta})\\
&\geq m\, E_{\sqrt m\, \e}(u_{\sqrt m\, \e}) \\ &\geq -m\pi d\, \Log(\sqrt m\, \e)-mC'
\end{split}\end{equation*}
and
\begin{equation*}\begin{split}
E_{\e,\delta}(u_{\e,\delta}) &\leq E_{\e,\delta}(u_{\sqrt M\,\e})\\
&\leq \frac{M}{2}\int_\Omega |\nabla u_{\sqrt M\,\e}|^2 +\frac{1}{4\e^2}\int_\Omega\big(1-|u_{\sqrt M\,\e}|^2\big)^2 \\
&=M E_{\sqrt M\,\e}(u_{\sqrt M\,\e}) \\
&\leq -M\pi d\, \Log(\sqrt M\,\e)+MC'\,.
\end{split}\end{equation*}
Hence the result, with $C=MC'+\max\big(m|\log(\sqrt m)|,M |\Log(\sqrt M)|\big)\,$.

\end{proof}

\medskip

\begin{lem}\label{en}
Let $\delta_n$ be any decreasing sequence (not necessarily tending to $0$). Then, there exists a sequence $\e_n$ satisfying the two properties\,:
\begin{itemize}
\item [1.] $\lim_{n\to\infty} \e_n/\delta_n^2 =0$\,;
\item[2.] $\displaystyle \frac{1}{\e_n^2}\int_\Omega\big(1-|u_{\e_n,\delta_n} |^2\big)^2 \leq 4M\pi  d$\,, $\forall \,n$\,.
\end{itemize}
\end{lem}
\begin{proof} The proof is based on an idea of M. Struwe [St1]. Fix $\delta>0$. Then, the function $\e\to \nu_{\e,\delta}=E_{\e,\delta}(u_{\e,\delta})$ is decreasing, hence almost surely differentiable, with
\begin{equation}\label{ineqint}
\nu_{\e_1,\delta}-\nu_{\e_0,\delta} \geq -\int_{\e_1}^{\e_0} \frac{\partial \nu_{\e,\delta}}{\partial \e}d\e\qquad (\e_1<\e_0)\,.
\end{equation}
(Note that one can show that $\e\to\nu_{\e,\delta}$ is a continuous map, and that consequently the above inequality is an equality).

Fix $\delta>0$. For given $\e'>\e$, compute
\begin{equation*}\begin{split}
\nu_{\e,\delta}-\nu_{\e',\delta} &= E_{\e,\delta}(u_{\e,\delta})-E_{\e',\delta}(u_{\e',\delta}) \\
&\geq E_{\e,\delta}(u_{\e,\delta})-E_{\e',\delta}(u_{\e,\delta}) \\
&=\left(\frac{1}{4\e^2}-\frac{1}{4\e'^2}\right)\int_\Omega\big(1-|u_{\e,\delta}|^2\big)^2\,,
\end{split}\end{equation*}
which, dividing by $\e-\e'$\,, implies, as $\e'\to \e$\,:
$$\frac{\partial \nu_{\e,\delta}}{\partial \e}\leq -\frac{1}{2\e^3}\int_\Omega\big(1-|u_{\e,\delta}|^2\big)^2\,.$$

From the latter, we deduce $\displaystyle \liminf_{\e\to 0} \frac{1}{\e^2}\int_\Omega\big(1-|u_{\e,\delta}|^2\big)^2\leq 2M\pi d$ $\big($\,if not, integrating between $\e_0$ and some $\e_1$ small enough and applying (\ref{ineqint}) would lead to a contradiction with (\ref{majmu})\,$\big)$. So, for given $\delta_n$, one will find  $\e_n<\delta_n^2/n$ such that $\displaystyle \frac{1}{\e_n^2}\int_\Omega\big(1-|u_{\e_n,\delta_n}|^2\big)^2\leq 4M\pi d\,$.

\end{proof}
\bigskip

\section{identifying the singular points}

\medskip \subsection {First elliptic estimates.} \-

\smallskip By (\ref{mM}) and equivalence of Sobolev norms, for $p\in [1,\infty]$, there will be a constant
$C_p>0$ such that
\begin{equation}\label{estellA}
 ||\Delta f||_p \leq C_p\big( ||div\,(A\nabla f)||_p + ||\nabla f||_p + ||f||_p\big)\, ,\; f\in W^{2,p}(\R^2)\,.
\end{equation}

Applying (\ref{estellA}) to the function $x\to f(\delta x)$, we get
\begin{equation}\label{estellAdelta}
||\Delta f||_p \leq C_p\left( ||div\,(A_\delta \nabla f)||_p + \frac{1}{\delta} ||\nabla f||_p + \frac{1}{\delta^2} ||f||_p\right),\; f\in W^{2,p}(\R^2).
\end{equation}

\bigskip
\subsection{Locating the singularities.} The above estimates allow us to follow what [BBH] calls the {\it construction of bad disks}, in order to get

\medskip
\begin{prop}\label{baddisks} Let $\delta_n$ be a decreasing sequence and let $\e_n$ a sequence associated to it by Lemma \ref{en}. Then, replacing $(\delta_n,\,\e_n)$ by a subsequence, one will find
$ \lambda>0$, $ N\in \N^*$, $\,a_1,\cdots,a_N\in \Omega$ such that, for any $n>0$, one has
\begin{equation}\label{bondisque}
|u_n(x)|\geq 1/2\,,\;\forall x \in \, \Omega \backslash \left(\cup_{i=1}^N B(a_i,\lambda \e_n)\right)\,.
\end{equation}
\end{prop}

\begin{proof} By Lemma \ref{uborne} and (\ref{Euler}), one has, for all $\e$ and $\delta$ (since $|u_{\e,\delta}|\leq 1$),
\begin{equation}
||div\,\big(A_\delta \nabla u_{\e,\delta}\big)||_\infty \leq \frac{1}{\e^2}\,,
\end{equation}
and consequently, by (\ref{estellAdelta})\,:
\begin{equation}\label{estDeltauinfini}
||\Delta u_{\e,\delta}||_\infty \leq C_\infty\left( \frac{1}{\e^2}+ \frac{1}{\delta}||\nabla u_{\e,\delta}||_\infty + \frac{1}{\delta^2}\right)
\end{equation}.

By Lemma \ref{en}, one has $\displaystyle \frac{1}{\delta_n}=o\big(\frac{1}{\e_n}\big)$.  From (\ref{estDeltauinfini}) above and the interpolation inequality $||\nabla u||_\infty^2 \leq ||\Delta u||_\infty\,||u||_\infty$ (cf. [BBH2])\,,
one deduces easily the existence of a constant $C$ not depending on $n$ such that
\begin{equation}\label{nablauinfini}
||\nabla \un||_\infty \leq \frac{C}{\e_n}\,.
\end{equation}

From this point, invoke Lemma \ref{en} and (\ref{nablauinfini}) and follow the steps for constructing the bad disks in [BBH], chapters III and IV.

\end{proof}

\medskip
\bigskip Next proposition  tells us that, in the neighborhood of a singular point $a_i$, the $\C$-valued map $u_{\e_n,\delta_n}$ behaves approximately as the $S^1$-valued map $u_{\e_n,\delta_n}/|u_{\e_n,\delta_n}|$\,:

\smallskip
\begin{prop}\label{usuru} Let $\delta_n$, $\e_n$, $\lambda$, $a_1,\cdots,a_N$, be as in Proposition \ref{baddisks}. Let $R_0>0$ be such that the disks $B(a_i,2R_0)$ are contained in $\Omega$ and do not intersect each other. Then there exist a constant $C$ such that, for any $R\in]0,R_0]$ and any $n$ with $\lambda \e_n<R$, one has
\begin{equation}\label{inequsuru}
0<
\int_{\Gamma_n^i(R)} \nabla\frac{ u_{\e_n,\delta_n}}{|u_{\e_n,\delta_n|}}\cdot A_{\delta_n} \nabla\frac{ u_{\e_n,\delta_n}}{|u_{\e_n,\delta_n|}} -
\int_{\Gamma_n^i(R)} \nabla u_{\e_n,\delta_n}\cdot A_{\delta_n} \nabla u_{\e_n,\delta_n}
\leq C\,,
\end{equation}
where $\Gamma_n^i(R)$ is the annulus $$\Gamma_n^i(R)=\{x\in \Omega\,/\,\lambda \e_n\leq |x-a_i|\leq R\}\,: $$

\end{prop}

\begin{proof} The proof follows the proof of Theorem ??  in [BBH], where we introduce explicitly the necessary adaptions.

Set $u_n=u_{e_n,\delta_n}$, $|u_n|=\rho_n$ and $v_n=u_n/\rho_n$. Then, one has
\begin{equation}\begin{split}\label{nablavn}
\nabla u_n.A_{\delta_n} \nabla u_n&=\rho_n^2\,\nabla v_n A_{\delta_n} \nabla v_n + \nabla \rho_n. A_{\delta_n} \nabla \rho_n\\
&\geq \rho_n^2\,\nabla v_n A_{\delta_n} \nabla v_n \\
&= \nabla v_n A_{\delta_n}{\delta_n} \nabla v_n - (1-\rho_n^2) \nabla v_n.A_{\delta_n} \nabla v_n \,.
\end{split}\end{equation}
Moreover, $\displaystyle
\nabla v_n.A_{\delta_n} \nabla v_n\leq \frac{1}{\rho_n^2}\nabla u_n.A_{\delta_n} \nabla u_n \leq 4M |\nabla u_n|^2$ (since $\rho_n\geq 1/2$ on $\Gamma_n^i(R)$). So that one has
\begin{equation*}\begin{split}
0<
\int_{\Gamma_n^i(R)} &\nabla\frac{ u_{\e_n,\delta_n}}{|u_{\e_n,\delta_n|}}.A_{\delta_n} \nabla\frac{ u_{\e_n,\delta_n}}{|u_{\e_n,\delta_n|}} -
\int_{\Gamma_n^i(R)} \nabla u_{\e_n,\delta_n}.A_{\delta_n} \nabla u_{\e_n,\delta_n}
\end{split}\end{equation*}
and
\begin{equation*}\begin{split}
\int_{\Gamma_n^i(R)} &\nabla\frac{ u_{\e_n,\delta_n}}{|u_{\e_n,\delta_n|}}.A_{\delta_n} \nabla\frac{ u_{\e_n,\delta_n}}{|u_{\e_n,\delta_n|}} -
\int_{\Gamma_n^i(R)} \nabla u_{\e_n,\delta_n}.A_{\delta_n} \nabla u_{\e_n,\delta_n} \\
&\leq 4M\int_{\Gamma_n^i(R)} \big(1-|u_n|^2\big)\,|\nabla u_n|^2 \\
& \leq 4M\,||1-|u_n|^2||_{L^2(\Omega)}\,||\nabla u_n||_{L^4(\Gamma_n^i(R)}^2\,.
\end{split}\end{equation*}

By (\ref{Euler}) and Lemma \ref{en}, one has $||\div\,\big(A_\delta \nabla\un\big)||_2 \leq \sqrt{4\pi Md}/\e_n$\,. Together with (\ref{majmu}), (\ref{estellAdelta}) and the fact that $\e_n=o(\delta_n^2)$, it implies easily $||\Delta \un||_2 \leq C/\e_n$ for some constant $C$ not depending on $n$.

By Property 2. of Lemma \ref{en}, one has $||1-|u_n|^2||_{L^2}\leq C\e_n$\,; and by the interpolation inequality $||\nabla u||_4^2 \leq ||\Delta u||_2\,||u||_\infty$, one has
$||\nabla u_n||_{L^4}^2\leq \sqrt{C}/\e_n$. Hence the result.

\end{proof}

\bigskip

\section{Energy of $S^1$-valued maps on an annulus\,: \\ General properties}

\bigskip

At this point, begins the second (and more difficult) part of this paper. It consists in comparing the energy of $u_{\e_n,\delta_n}$ on an annulus $\Gamma_n^i(R)$, with the minimal energy of a $S^1$-valued map on the same annulus with the same degree. As the degree of the boundary value of an $H^1$-map is not continuous under weak $H^1$-limits, we shall restrict ourselves to a class of $S^1$-values $H^1$-maps where this difficulty is naturally overcome.

Note that $u_{\e_n,\delta_n}$ is of class $H^2$ (by (\ref{Euler})), and consequently, one can write, on an annulus $\Gamma_n^i(R)$ of Proposition \ref{usuru},
$$\frac{u_{\e_n,\delta_n}}{|u_{\e_n,\delta_n}|}(x)=
e^{if(r,\theta)}\,\text{for}\, x=a_i+(r\cos\theta,r\sin\theta)\,,\;
r\in [\lambda\e_n,R], \theta\in [0,2\pi]$$
with $f\in H^2([\alpha,\beta]\times[0,2\pi],\R)$ and $f(r,2\pi)=f(r,0)+2\pi\kappa_i(n)$ for some $\kappa_i(n)\in\Z$ depending only on $a_i$ and $n$\,.

 We shall restrict ourselves to a class $\mathcal V_\kappa$ of  $S^1$-valued map of degree $\kappa\in\Z$ sharing a similar property.

\medskip In this section, we shall consider a field $$\R^2\ni x \to B(x)\in M_2(\R)$$ of symmetric positive definite
$2\times 2$-matrices, of class $W^{2\infty}$, with\\ $spectrum(B(x))\subset [m,M]$ for some $0<m<M$ not depending of $x$. The notation $\displaystyle
\widetilde B(r,\theta)=
\begin{pmatrix} \widetilde B_{rr}(r,\theta) & \widetilde B_{r\theta}(r,\theta) \\ \widetilde B_{r\theta}(r,\theta) & \widetilde B_{\theta\theta}(r,\theta)
\end{pmatrix}
$ will stand for $B(r\cos\theta,r\sin\theta)$, written in the orthonormal basis $(\partial/\partial r, 1/r\,\partial/\partial \theta)$\,.

\medskip
Let us
 introduce the following notations\,:
\begin{notations}\label{defdegkappa} For  $\alpha,\beta$, $0<\alpha<\beta$, and $\kappa\in \Z$\,, define
$$\Gamma(\alpha,\beta)=\{x\in\R^2\,/\, \alpha\leq |x| \leq \beta\}\,;$$
\begin{equation*}\begin{split}
\mathcal V_\kappa(\alpha,\beta)=\Big\{ v\in H^1(\Gamma(\alpha,\beta&),S^1)\, \Big / \,\exists \, f\in H^1([\alpha,\beta]\times [0,2\pi],\R) \text{ s.t.}\\
f(r,2&\pi)=f(r,0)+2\kappa \pi \,,\,\forall \,r\in[\alpha,\beta]\text{ a.s. and }
\\ &v(r,\theta)=e^{if(r,\theta)}\,,\, \alpha\leq r\leq \beta\,,\, 0\leq \theta\leq 2\pi \Big\}\,;
\end{split}\end{equation*}

\smallskip
\begin{equation*}\begin{split}
\mu(B,\alpha,\beta,\kappa)
= \inf\, \displaystyle \int_{\Gamma(\alpha,\beta)} \nabla v(x).B(x)\,\nabla v(x)\,dx\,,\; v\in \mathcal V_\kappa (\alpha,\beta)\,.
\end{split}\end{equation*}

\end{notations}

\medskip The first properties of the $\mu(B,\alpha,\beta,\kappa)$ are more or less obvious\,:
\begin{lem}\label{kappa} One has $\mu(B,\alpha,\beta,\kappa)=\kappa^2 \mu(B,\alpha,\beta,1)$.
\end{lem}

\smallskip
\begin{lem}\label{estmu0} One has
\begin{equation}\label{estmu}
2m\pi \kappa^2 \Log\left(\frac{\beta}{\alpha}\right) \leq \mu(B,\alpha,\beta,\kappa) \leq 2M\pi \kappa^2 \Log\left(\frac{\beta}{\alpha}\right)\,.
\end{equation}
\end{lem}

Before proving those two lemmas, let us introduce some additional notations\,:
\begin{notations}
$$\Gamma'(\alpha,\beta)=[\alpha,\beta]\times[0,2\pi]\,;$$
\begin{equation*}\begin{split}
\mathcal D_\kappa(\alpha,\beta)=\big\{f=f(r&,\theta)\in\, H^1(\Gamma'(\alpha,\beta),\R)\,/\,  \\
&f(r,2\pi)=f(r,0)+2\kappa \pi\,,\;r\in[\alpha,\beta] \text{ a.s.} \big\}\,;
\end{split}\end{equation*}
$$Df(r,\theta)=\begin{pmatrix}\displaystyle \frac{\partial f}{\partial r}\\ \\ \displaystyle \frac{1}{r}\frac{\partial f}{\partial \theta}\end{pmatrix}\,,\; f\in H^1(\Gamma'(\alpha,\beta),\R)\,.$$
\end{notations}

\smallskip
Note that, for $v(r,\theta)=e^{if(r,\theta)}\in \mathcal V_\kappa$ as above, one has
\begin{equation}
\int_{\Gamma(\alpha,\beta)}\nabla v(x).B(x)\nabla v(x)\,dx = \int_{\Gamma'(\alpha,\beta)}Df(r,\theta). \widetilde B(r,\theta)Df(r,\theta)\,rdrd\theta\,,
\end{equation}
so that
 \begin{equation}\label{mukappa1}\begin{split}
\mu(B,\alpha,\beta,\kappa)&=\inf\, \displaystyle \int_{\Gamma'(\alpha,\beta)} Df.\widetilde B(r,\theta)\,Df\,rdrd\theta\,,\; f\in \mathcal D_\kappa (\alpha,\beta)
\end{split}\end{equation}

\medskip\noindent {\it Proof of Lemma \ref{kappa}.} For $\kappa=0$, take $v$ a constant function. Otherwise, note that
$f\in \mathcal D_\kappa(\alpha,\beta)$ if and only if $\displaystyle \frac{1}{\kappa}f \in \mathcal D_1(\alpha, \beta)$\,, and apply (\ref{mukappa1}).

\medskip\noindent {\it Proof of Lemma \ref{estmu0}.} For $f\in \mathcal D_\kappa(\alpha,\beta)$, one has $\displaystyle \int_0^{2\pi} \frac{\partial f(r,\theta)}{\partial \theta}d\theta=2\kappa \pi$, which, by Cauchy-Schwartz, implies $\displaystyle \int_0^{2\pi} \left(\frac{\partial f(r,\theta)}{\partial \theta}\right)^2d\theta\geq 2\pi \kappa^2$. Hence
\begin{equation*}\begin{split}
\int_{\Gamma'(\alpha,\beta)}Df.\widetilde B(r,\theta)Df \,rdrd\theta &\geq m\int_{\Gamma'(\alpha,\beta)} |Df(r,\theta)|^2 rdrd\theta \\
 &\geq m \int_\alpha^\beta \frac{dr}{r}\int_0^{2\pi}  \left(\frac{\partial f(r,\theta)}{\partial \theta}\right)^2 d\theta \\
 &\geq 2m\pi\kappa^2 \Log\left( \frac{\beta}{\alpha}\right)
\end{split}\end{equation*}
which provides the first inequality.

\smallskip
Taking $f(r,\theta)=\kappa \theta$ provides the second inequality.

\hfill $\square$

\medskip
\begin{prop}\label{minimiseur} Let $\alpha,\beta,\kappa$ be fixed.

1/ There exists  $v_{\alpha,\beta}\in \mathcal V_\kappa(\alpha,\beta)$ such that $\displaystyle \int_{\Gamma(\alpha,\beta)} \nabla v_{\alpha,\beta}(x). B(x)\nabla v_{\alpha,\beta}(x)dx =\mu(B,\alpha,\beta,\kappa)
$\,.

\smallskip
2/ Such a minimizer $v_{\alpha,\beta}$ is unique, up to a multiplicative modulus $1$ constant, and is of class $H^2$.

\smallskip
3/  The Euler-Lagtange equation satisfied by $v_{\alpha,\beta}$ is the equation of $S^1$-valued harmonic functions
\begin{equation}\label{Eulerv}
-div\,\big(B(x)\nabla v(x)\big) = v(x)\,\big( \nabla v(x)\cdot B(x)\nabla v(x)\big)\,,
\end{equation}
together with boundary conditions of Neumann type
\begin{equation}\label{bordv1}
\frac{\partial v(\alpha,\theta) }{\partial r}= \frac{\partial v(\beta,\theta) }{\partial r}=0\,,\; \forall\,\theta\in [0,2\pi] \text{ a.s.}\,.
\end{equation}

\end{prop}

\begin{proof} By (\ref{mukappa1}), the problem consists in studying minimizers in $\mathcal D_\kappa(\alpha,\beta)$ of the quadratic form $\displaystyle \int_{\Gamma'(\alpha,\beta)}Df(r,\theta)\cdot \widetilde B(r,\theta)Df(r,\theta)rdrd\theta$\,.

1/ For the existence  of a minimizer $f_{\alpha,\beta}$, it is enough to note that $\mathcal D_\kappa$ is a closed affine subspace of $H^1(\Gamma'(\alpha,\beta),\R)$ (directed by $\mathcal D_0(\alpha,\beta)$\,, which is a closed vector subspace), hence weakly closed, and that the quadratic form to minimize on $\mathcal D_\kappa$ is l.s.c. for the weak $H^1$-topology.

\smallskip 2/ and 3/ The Euler-Lagrange Equation satisfied by $f=f_{\alpha,\beta}$ is
\begin{equation}\label{orthD0}
\int_{\Gamma'(\alpha,\beta}Dg(r,\theta).\widetilde B(r,\theta)Df(r,\theta)\,rdrd\theta=0\,,\; \forall\, g \in \mathcal D_0(\alpha,\beta)\,.
\end{equation}

Taking $g\in C^{\infty}_c(]\alpha,\beta[\times ]0,2\pi[)$ ($C^{\infty}$-functions with compact support) provides
\begin{equation}\label{Eulerf}
-div_D\big(\widetilde B(r,\theta)Df(r,\theta)\big)=0\,,
\end{equation}
where $-div_D$ is the formal adjoint of $D$, i.e.
$$div_D(X)=\frac{1}{r}\frac{\partial (rX_r)}{\partial r}+ \frac{1}{r} \frac{\partial X_\theta}{\partial \theta}\,,\; X\in H^1(\Gamma'(\alpha,\beta),\R^2)\,.
$$
The corresponding equation for $v\,:\;v(r,\theta)=e^{if(r,\theta)}$ is (\ref{Eulerv}).

\smallskip Taking $g\in C^{\infty}_c([\alpha,\beta]\times ]0,2\pi[)$, integrating by parts and invoking (\ref{Eulerf}) provides
\begin{equation}
\frac{\partial f(\alpha,\theta) }{\partial r}= \frac{\partial f(\beta,\theta) }{\partial r}=0\,,\; \forall\,\theta\in [0,2\pi] \text{ a.s.}\,,
\end{equation}
hence (\ref{bordv1}).

Taking  $g\in C^{\infty}_c(]\alpha,\beta[\times [0,2\pi])$ and again invoking (\ref{Eulerf}) provides
\begin{equation}\label{bordf2}
\frac{\partial f(r,2\pi)}{\partial \theta}=\frac{\partial f(r,0)}{\partial \theta}\,,\; \forall \, r\in [\alpha,\beta] \text{ a.s.}\,.
\end{equation}
\smallskip (\ref{Eulerf}) implies that $f_{\alpha,\beta}$ is of class $H^2$. This fact, together with (\ref{bordf2}), will imply that $v$ is of class $H^2$.
\medskip

For the uniqueness of $v$ up to a multiplicative constant, note that, if $f_1$ and $f_2$ are two minimizers in $\mathcal D_\kappa(\alpha,\beta)$, one has, setting $g=f_2-f_1$ and invoking (\ref{orthD0})\,:
\begin{equation*}\begin{split}
\int_{\Gamma'(\alpha,\beta)}&Df_1.\widetilde B(r,\theta)Df_1 \,rdrd\theta = \int_{\Gamma'(\alpha,\beta)}Df_2.\widetilde B(r,\theta)Df_2 \,rdrd\theta \\
&=\int_{\Gamma'(\alpha,\beta)}Df_1.\widetilde B(r,\theta)Df_1 \,rdrd\theta + \int_{\Gamma'(\alpha,\beta)}Dg.\widetilde B(r,\theta)Dg \,rdrd\theta \\
&\geq \int_{\Gamma'(\alpha,\beta)}Df_1.\widetilde B(r,\theta)Df_1 \,rdrd\theta+m\int_{\Gamma'(\alpha,\beta)} |Dg|^2 rdrd\theta
\end{split}\end{equation*}
which is possible if and only if $Dg=0$, i.e. iff $g=f_2-f_1$ is a constant function.

\end{proof}

\section{Energy of $S^1$-valued maps on an annulus\,: \\ behaviour at the boundary}

\bigskip Estimates on the value of the minimal energy $\mu(B,\alpha,\beta,\kappa)$ could be provided, but the main feature here is that such estimates are not needed when one seeks only to prove the existence  of a limit configuration for the $\un$. What is actually needed is the existence of some ''approximate minimizers'' for which the value at the boundary of the annulus is prescribed. As it will appear in the proof, any a priori value could have been prescribed. But we shall limit our study to the standard boundary condition of degree $\kappa$, namely $\theta\to e^{i\kappa \theta}$.

\smallskip
\begin{thm}\label{boundval} Fix $m$ and $M$, $0<m<M$\,. There exists a constant $C$, depending only on $m$ and $M$, such that, for any field $B(.)$ of symmetric $2\times 2$-matrices with $spec(B(x)\subset [m,M]$, $\forall\,x\in \R^2$, and for any
$\alpha,\beta,\kappa$\,, one can find $v\in \mathcal V_\kappa(\alpha,\beta)$ with the two properties\,:
\begin{itemize}
\item [1/] $\displaystyle \int_{\Gamma(\alpha,\beta)} \nabla v(x)\cdot B(x)\nabla v(x) \leq \mu(B,\alpha,\beta,\kappa)+\kappa^2C\,$:

\smallskip
\item [2/] $\exists \,\theta_0$\,,
$v(\beta,\theta)=e^{i\kappa \theta}\,,\; v(\alpha,\theta)=e^{i\kappa(\theta+\theta_0)}$, $\theta\in [0,2\pi]\,. $
\end{itemize}
\end{thm}

\begin{proof} Without lack of generality, we shall only consider pairs $(\alpha,\beta)$ such that $\beta >4 \alpha$ (the case $\beta\leq 4\alpha$ is solved by Lemma \ref{estmu0} and its proof, taking $v(r,\theta)=e^{i\kappa\theta}$\,).

By Lemma \ref{kappa}, we can restrict ourselves to the case $\kappa=1$ (if $v$ satisfies the conclusions of the Theorem for $\kappa=1$, then $v^{\kappa}$ will be a solution for a general $\kappa\in \Z$).

Let
$v_1=v_{2\alpha,\beta/2}$ be a minimizer in $\mathcal V_1(2\alpha,\beta/2)$ for $\mu(B,2\alpha,\beta/2,1)$ provided by Proposition \ref{minimiseur}, with $v_1(r,\theta)=e^{if_1(r,\theta)}$, $f\in \mathcal D_1(2\alpha,\beta/2)$. Set
\begin{equation}\begin{split}
J=\Big\{ r\in[2\alpha,\beta/2]\,\Big/ \, \int_0^{2\pi} Df_1(r,\theta)&\cdot \widetilde B(r,\theta)Df_1(r,\theta)d\theta \\
&\leq \frac{1}{r^2}\int_0^{2\pi} \widetilde B_{\theta \theta}(r,\theta)d\theta\,\Big\}\,.
\end{split}\end{equation}
 As $f_1$ is of class $H^2$ (cf. Proposition \ref{minimiseur}), $J$ is a closed subset of $[2\alpha,\beta/2]$.

 \smallskip We claim that $J\not=\emptyset$\,. If this was not true, one would define $g_1\in \mathcal D_1(2\alpha,\beta/2)$ by $g_1(r,\theta)=\theta$, and get, by definition of $J$\,:
\begin{equation}\begin{split}
\int_{\Gamma'(2\alpha,\beta/2)}&Dg_1(r,\theta)\cdot \widetilde B(r,\theta)Dg_1(r,\theta)d\theta \\
&= \int_{2\alpha}^{\beta/2}\frac{dr}{r}\int_0^{2\pi} \widetilde  B_{\theta\theta}(r,\theta)d\theta \\
& < \int_{2\alpha}^{\beta/2}rdr\int_0^{2\pi}Df_1(r,\theta)\cdot \widetilde B(r,\theta)Df_1(r,\theta)d\theta \\
&= \mu(B,2\alpha,\beta/2,1)\,,
\end{split}\end{equation}
which would contradict Formula (\ref{mukappa1}).

\smallskip So, we can define $r_1=\inf J$ and $r_2=\sup J$\,. Note that, for $r\in J$, one has
\begin{equation}\label{rJ}\begin{split}
\int_0^{2\pi} \left| \frac{\partial f_1(r,\theta)}{\partial \theta} \right|^2 d\theta & \leq
r^2 \int_0^{2\pi} |Df_1(r,\theta)|^2 d\theta \\
&\leq \frac{r^2}{m}\int_0^{2\pi} Df_1(r,\theta)\cdot \widetilde B(r,\theta)Df_1(r,\theta)d\theta \\
&\leq \frac{1}{m}\int_0^{2\pi} \widetilde B_{\theta,\theta}(r,\theta)d\theta\,.
\end{split}\end{equation}
In particular, as $m\leq \widetilde  B_{\theta,\theta}(r,\theta)\leq M$, one has
\begin{equation}\label{r1r2}
\int_0^{2\pi} \left| \frac{\partial f_1(r,\theta)}{\partial \theta} \right|^2 d\theta \leq \frac{2\pi M}{m}\;\text{ for } r=r_1 \text{ or } r=r_2\,.
\end{equation}

\smallskip As $v_{2\alpha,\beta/2}$ is defined up to a multiplicative modulus $1$ constant, we can suppose that $f_1(r_2,0)=0$.
With this choice of $f_1$, we set $\theta_0=f_1(r_1,0)$ and we define $f\in \mathcal D_1(\alpha,\beta)$ by the formula
\begin{equation}
f(r,\theta) = \left\{ \begin{matrix} \theta & \text{ if } 2r_2\leq r \leq \beta\,; \\ \\ \displaystyle \frac{2r_2-r}{r_2} f_1(r_2,\theta) + \frac{r-r_2}{r_2}\theta & \text{ if } r_2\leq r \leq 2r_2\,; \\ \\
f_1(r,\theta) &\text{ if } r_1\leq r \leq r_2 \,;\\ \\
 \displaystyle 2\frac{r_1-r}{r_1} f_1(r_1,\theta) + \frac{2r-r_1}{r_1}(\theta+\theta_0) & \text{ if } r_1/2\leq r \leq r_1\,; \\ \\
 \theta+\theta_0 &\text{ if } \alpha \leq r \leq r_1/2\,.
\end{matrix}\right.
\end{equation}
We will check that, with such an $f$, $v(r,\theta)=e^{if(r,\theta)}$ satisfies the conclusions of the theorem. Note that $v\in \mathcal V_1(\alpha,\beta)$ by construction.

\smallskip Let $\lambda_1$ be the first eigenvalue of the Laplacian in $H^1_0(\,]\,r_2,2r_2\,]\,)$ (which as a matter of fact does not depend on the choice of $r_2>0$). We have, for $r\in [r_2,2r_2]$\,:
\begin{equation}\label{estDf1}\begin{split}
\int_0^{2\pi} \left| \frac{\partial f(r,\theta)}{\partial r}\right|^2 d\theta &=\frac{1}{r_2^2} \int_0^{2\pi} \big| \theta - f_1(r_2,\theta) \big|^2 d\theta \\
&\leq \frac{\lambda_1}{r_2^2} \int_0^{2\pi} \left| 1-\frac{\partial f_1(r_2,\theta)}{\partial \theta} \right|^2 d\theta \\
&\leq \frac{2\lambda_1}{r_2^2} \int_0^{2\pi}\left(1+\left| \frac{\partial f_1(r_2,\theta)}{\partial \theta} \right|^2 \right) d\theta \\
&\leq \frac{4\pi \lambda_1}{r_2^2}\left( 1+\frac{M}{m}\right)  \,,
\end{split}\end{equation}
the last inequality being provided by (\ref{r1r2}).

We compute then, invoking again (\ref{r1r2})\,:
\begin{equation}\label{estDf2}\begin{split}
\int_0^{2\pi} \left| \frac{\partial f(r,\theta)}{\partial \theta}\right|^2 d\theta &=
\int_0^{2\pi} \left( \frac{2r_2-r}{r_2}\, \frac{ \partial f_1(r_2,\theta) }{\partial \theta}+ \frac{r-r_2}{r_2}\right)^2 d\theta \\
&\leq \int_0^{2\pi} \left( \frac{2r_2-r}{r_2} \left|\frac{ \partial f_1(r_2,\theta) }{\partial \theta}\right|^2+ \frac{r-r_2}{r_2}\right) d\theta \\
&\leq \frac{4\pi M}{m}\,.
\end{split}\end{equation}

Integrating  Inequalities (\ref{estDf1}) and (\ref{estDf2}) between $r_2$ and $2r_2$, we get
\begin{equation}\begin{split}
\int_{\Gamma'(r_2,2r_2)}& Df(r,\theta)\cdot \widetilde B(r,\theta) Df(r,\theta) \,rdrd\theta \\ & \leq M \int_{\Gamma'(r_2,2r_2)} \big|Df(r,\theta)\big|^2 rdrd\theta \\
&= M \int_{r_2}^{2r_2} rdr \int_0^{2\pi} \left| \frac{\partial f(r,\theta)}{\partial r}\right|^2 d\theta +
M \int_{r_2}^{2r_2} \frac{dr}{r} \int_0^{2\pi} \left| \frac{\partial f(r,\theta)}{\partial \theta}\right|^2 d\theta \\
&\leq 6\pi \lambda_1M \left( 1+\frac{m}{M}\right) + 4\pi \frac{M^2}{m}\Log 2\,.
\end{split}\end{equation}
Writing, for sake of simplicity
\begin{equation}
C_1(m,M)=6\pi \lambda_1M \left( 1+\frac{m}{M}\right) + 4\pi \frac{M^2}{m}\Log 2\,,
\end{equation}
a similar computation leads to
\begin{equation}\begin{split}
\int_{\Gamma'(r_1/2,r_1)} Df(r,\theta)\cdot \widetilde B(r,\theta) Df(r,\theta) \,rdrd\theta \leq C_1(m,M)\,.
\end{split}\end{equation}

\smallskip Suppose now that $2r_2< \beta/2$. We have then $]2r_2,\beta/2]\cap J =\emptyset$, which, by definition of $J$, will imply
\begin{equation}\begin{split}
\int_{\Gamma'(2r_2,\beta/2)} Df(r,\theta)& \cdot \widetilde B(r,\theta) Df(r,\theta)rdrd\theta = \int_{2r_2}^{\beta/2} \frac{dr}{r} \int_0^{2\pi} \widetilde B_{\theta\theta}(r,\theta)d\theta \\
&\leq  \int_{2r_2}^{\beta/2} rdr\int_0^{2\pi} Df_1(r,\theta)\cdot \widetilde B(r,\theta) Df_1(r,\theta)d\theta \\
&=\int_{\Gamma'(2r_2,\beta/2)} Df_1(r,\theta)\cdot \widetilde B(r,\theta) Df_1(r,\theta)d\theta\,,
\end{split}\end{equation}
while
\begin{equation}\begin{split}
\int_{\Gamma'(\beta/2,\beta)} Df(r,\theta) \cdot \widetilde B(r,\theta) Df(r,\theta)rdrd\theta &= \int_{\beta/2}^{\beta} \frac{dr}{r} \int_0^{2\pi} \widetilde B_{\theta\theta}(r,\theta)d\theta \\
&\leq M\Log 2\,.
\end{split}\end{equation}

In the case where $2r_2\geq \beta/2$, we write only
\begin{equation}\begin{split}
\int_{\Gamma'(2r_2,\beta)} Df(r,\theta) \cdot \widetilde B(r,\theta) Df(r,\theta)rdrd\theta &= \int_{2r_2}^{\beta} \frac{dr}{r} \int_0^{2\pi} \widetilde B_{\theta\theta}(r,\theta)d\theta \\
&\leq M\Log 2\,.
\end{split}\end{equation}

The same computations, with the same upper bounds, hold on $\Gamma'(\alpha,r_1/2)$ or on $\Gamma'(\alpha,2\alpha)\cup \Gamma'(2\alpha,r_1/2)$. Setting $\widetilde r_1=r_1$ if $r_1/2\geq 2\alpha$ (resp. $\widetilde r_1=4\alpha$ if $r_1/2<2\alpha$), and $\widetilde r_2=r_2$ if $2r_2\leq \beta/2$ (resp. $\widetilde r_2=\beta/4$ if $2r_2>\beta/2$), we get finally
\begin{equation}\begin{split}
\int_{\Gamma'(\alpha,\beta)}& Df(r,\theta) \cdot \widetilde B(r,\theta) Df(r,\theta)rdrd\theta \\
&\leq \int_{\Gamma'(2\alpha,\widetilde r_1/2)\cup \Gamma'(r_1,r_2)\cup\Gamma'(2\widetilde r_2,\beta/2)
} Df_1(r,\theta) \cdot \widetilde B(r,\theta) Df_1(r,\theta)rdrd\theta \\
&\qquad\qquad\qquad\qquad\qquad\qquad\quad
+ 2C_1(M,m)+2M\Log 2 \\
&\leq \mu(B,2\alpha,\beta/2,1) + 2C_1(m,M)+2M\Log 2\,,
\end{split}\end{equation}
which implies the result.

\end{proof}

\newpage

\section{Existence of a limit configuration}

\bigskip
\subsection{Statement of the Theorem}\-

The data are those of subsection \ref{data}\,: the domain $\Omega$, the boundary condition $g:\Omega\to S^1$ of degree $d>0$, the $\Z^2$-periodic field $A(.)$ of $2\times 2$ positive definite matrices, of class $W^{2,\infty}$, with spectrum bounded above by $M$ and below by $m$ independently of $x$, and for every pair $(\e>0,\delta>0)$, a minimizer $u_{\e,\delta}$ of the Ginzburg-Landau energy $E_{\e,\delta}$ for the geometry provided by the field $\displaystyle A_\delta\,:\; A_\delta(x)=A\left( \frac{x}{\delta} \right)$ (cf. (\ref{defE}) ).

The whole section will be devoted to the proof of the following theorem, which states that (up to substituting a subsequence), for any decreasing  sequence $\delta_n$, there exists a sequence $\e_n$ such that the sequence $\un$ is locally bounded in some $H^1_{loc}$-space of $\Omega\backslash\{finite\;set\;of \; points\}$.

\smallskip
\begin{thm}\label{borneloc} The notations are those of subsection \ref{data}.

Let $\delta_n$ be a decreasing sequence in $\R^*_+$. Then, substituting to it a subsequence (still denoted $\delta_n$), one can find
\begin{itemize}
\item [$\cdot$] a sequence $\e_n$ tending to $0$\,,
\item [$\cdot$] a finite subset $\{a_1,\cdots,a_N\}$ of $\Omega$\,,
\end{itemize}
such that the sequence $\un$ is bounded in $H^1_{loc}\big(\Omega\backslash\{a_1,\cdots,a_N\}\big)$.

Which means that, for any $R>0$, one has
\begin{equation}
\sup_n \int_{\Omega \backslash (\,\cup_j B(a_j,R)\,)} |\nabla \un|^2 < +\infty\,.
\end{equation}

\end{thm}

\bigskip\subsection{Proof of Theorem \ref{borneloc}}.\-

The sequence $\delta_n$ being given, we fix an associated sequence $\e_n$ provided by Lemma \ref{en}, so that Conclusions 1/ and 2/ of the lemma hold true.

Then, we replace $(\e_n,\delta_n)$ by a subsequence in order to get the conclusions of Proposition \ref{baddisks}. This provides the finite set $\{a_1,\cdots,a_N\}$ of singular points and a ratio $\lambda>0$ for the annulus  $\Gamma_n^i(R)$ of Proposition \ref{usuru}.

We claim that the sequence $(\un)$ is bounded in $H^1_{loc}(\Omega\backslash \{a_1,\cdots,a_N\}$\,.

\medskip We fix $R_0$ such that the disks $B(a_i,2R_0)$ are contained in $\Omega$ and do not intersect each other. For any $R\leq R_0$, $|\un|\geq 1/2$ on $\Gamma_n^i(R)$, so that $\un$ has a well defined degree $\kappa_i(n)$ on the annulus. The first claim is that $\kappa_i(n)$ cannot be too large.

\smallskip
\begin{lem}\label{majkappa} $\exists\,\kappa_0\in \N^*\,,\; |\kappa_i(n)|\leq \kappa_0\,,\; n\geq 1\,,\; i=1,\cdots,N\,.$
\end{lem}

\begin{proof}\-

By Inequalities (\ref{majmu}), (\ref{inequsuru}) and (\ref{estmu}), one has a constant $C$ such that
\begin{equation}\begin{split}
-M\pi d\, \Log\, \e_n +C &\geq E_{\e_n,\delta_n}(\un) \\
& \geq \int_{\Gamma_n^i(R)} \nabla \un\cdot A_\delta(x)\nabla \un \\
 & \geq \int_{\Gamma_n^i(R)} \frac{\nabla \un}{|\nabla \un}\cdot A_\delta(x) \frac{\nabla \un}{|\nabla \un|} -C \\
 &\geq \mu(A_\delta(x-a_i),\lambda\e_n,R,\kappa_i(n)) -C \\
 &\geq -m\pi\kappa_i(n)^2 \,Log \left(\frac{R}{\lambda \e_n}\right)-C\,.
 \end{split}\end{equation}

 Dividing by $-\Log\, \e_n$ and making $n\to \infty$ provides
 \begin{equation}
\limsup_{n\to\infty} \kappa_i(n)^2 \leq \frac{Md}{m}
\end{equation}
and the result.

\end{proof}

\smallskip We continue the proof of the theorem by contradiction. Suppose that the sequence $(\un)$ above is not $H^1$-locally bounded on $\Omega \backslash \{a_1,\cdots,a_N\}$. Then, there will be some $R\leq R_0$ and a subsequence -- still denoted $(\un)$\ -- such that
\begin{equation}\label{contrad}
\lim_{n\to\infty}\int_{\Omega\backslash (\,\cup_iB(a_i,R)\,)}|\nabla \un |^2 = +\infty\,.
\end{equation}

By Lemma \ref{majkappa}, substituting again a subsequence to the sequence $(\un)$, one can suppose that the degrees $\kappa_i(n)$ do not depend on $n$, i.e.
\begin{equation}\kappa_i(n)=\kappa_i\,,\;i=1,\cdots,N\,,\;n\geq 1\,.
\end{equation}
As $|\un|\geq 1/2$ on $\Omega\backslash (\,\cup_i B(a_i,R)\,)$, one has $\sum_{i=1}^N \kappa_i=d$. Hence, there will exist some  $w_0\in H^1(\Omega\backslash (\,\cup_iB(a_i,R)\,),\C)$ satisfying
\begin{equation}\begin{split} | w_0(x)|=1\,,\;&x \in \,\Omega\backslash (\,\cup_iB(a_i,R)\,)\,,\quad\\ w_0(x)=g(x)\,,\;x\in \partial \Omega\,,\quad w_0(x)&=\frac{(x-a_i)^{\kappa_i}}{| x-a_i|^{\kappa_i}}\,,\;i=1,\cdots,N\,,\;| x-a_i|=R\,.
\end{split}\end{equation}

With the help of Theorem \ref{boundval}, we extend this $w_0$ as a function in $H^1_g(\Omega,\C)$ in the following way:

Theorem \ref{boundval} provides a constant $C(m,M)$ and, for  given $i=1\cdots N$ and $n>0$,
a function $w_n^i\in H^1(\Gamma_n^i(R),S^1)$ with some constant $\theta_n^i$ such that
\begin{equation}\label{constwn}\begin{split} w_n^i(x)=\frac{(x-a_i)^{\kappa_i}}{|x-a_i|^{\kappa_i}}&\,\text{ if }|x-a_i|=R\,,\; \\w_n^ix) =  e^{i\kappa_i\theta_n^i} \frac{(x-a_i)^{\kappa_i}}{|x-a_i|^{\kappa_i}}&\,\text{ if }|x-a_i|=\lambda \e_n\,, \\
\int_{\Gamma_n^i(R)}\
\nabla w_n^i(x)\cdot A_\delta(x)\nabla w_n^i(x)dx&\leq \mu(B_n^i,\lambda\e_n,R,\kappa_i)+\kappa_i^2C(m,M)\,,
\end{split}\end{equation}
where $B_n^i$ is the field $\displaystyle x\to A\left(\frac{x-a_i}{\delta_n}\right)$. [We recall that $\Gamma_n^i(R)$ is the annulus $\lambda \e_n \leq |x-a_i|\leq R$\,.]

Define $W_n\in  H^1_g(\Omega,\C)$ as follows\,:
\begin{equation}W_n(x)=\left\{
\begin{matrix} w_0(x) &\text{ if } x\in \Omega\backslash (\,\cup_iB(a_i,R)\,)\,; \\
w_n^i(x) &\text{if } x\in \Gamma_n^i(R)\,,\;i=1,\cdots,N\,;\\
e^{i\kappa_i\theta_n^i}\displaystyle\frac{(x-a_i)^{\kappa^i}}{(\lambda\e_n)^{\kappa^i}} &\text{if }x\in B(a_i,\lambda \e_n)\,,\; i=1,\cdots,N\,.
\end{matrix}\right.\end{equation}

\medskip We compute
\begin{equation}\label{estWn}\begin{split}
E_{\e_n,\delta_n}(W_n) & =\undemi \int_{\Omega \backslash(\,  \cup_i B(a_i,R) \,)} \nabla w_0(x)\cdot A_\delta(x) \nabla w _0(x)\,dx\\
&+\sum_{i=1}^N \undemi\int_{\Gamma_n^i(R)} \nabla w_n^i(x).A_\delta(x)\nabla w_n^i(x) \,dx\\
&+\sum_{i=1}^N\ \undemi\int_{B(a_i,\lambda\e_n)}\nabla W_n(x).A_\delta(x)\nabla W_n(x)\,dx \\
&+ \sum_{i=1}^N
\frac{1}{4\e_n^2}\int_{B(a_i,\lambda\e_n)}\big( 1-|W_n(x)|^2\big)^2\,dx\,.
\end{split}\end{equation}

In the right hand side, the first line is
\begin{equation}\label{estW1}
\undemi \int_{\Omega \backslash(\,  \cup_i B(a_i,R) \,)} \nabla w_0\cdot A_\delta \nabla w _0 \leq \frac{M}{2} \int_{\Omega \backslash(\,  \cup_i B(a_i,R) \,)} \big| \nabla w_0\big|^2 =MC_0
\end{equation}
with $C_0$ not depending on $n$.

We have then, by (\ref{constwn})\,:
\begin{equation}\label{estW2}
\sum_{i=1}^N \int_{\Gamma_n^i(R)} \nabla w_n^i(x).A_\delta(x)\nabla w_n^i(x)  \leq \sum_{i=1}^N \mu(B_n^i,\lambda\e_n,R,\kappa_i) + N \kappa_0^2 C(m,M)\,.
\end{equation}

We continue with
\begin{equation}\label{estW3}\begin{split}
\sum_{i=1}^N\ \undemi\int_{B(a_i,\lambda\e_n)}&\nabla W_n(x).A_\delta(x)\nabla W_n(x) \\
& \leq \frac{M}{2} \sum_{i=1}^N \int_{B(0,\lambda \e_n)} \big| \nabla \big( \frac{re^{i\theta}}{\lambda\e_n}\big)^{\kappa_i} \big|^2 rdrd\theta\leq \frac{MN\kappa_0^2}{2}C_1
\end{split}\end{equation}
with $C_1$ not depending on $n$.

And finally, since $|W_n|\leq 1$\,:
\begin{equation}\label{estW4}
\sum_{i=1}^N
\frac{1}{4\e_n^2}\int_{B(a_i,\lambda\e_n)}\big( 1-|W_n(x)|^2\big)^2 \leq \frac{N\lambda^2 \pi}{4}\,.
\end{equation}

\smallskip Summing up in (\ref{estWn}), and using the fact that $\un$ is a minimizer for $E_{\e_n,\delta_n}$, we get
\begin{equation}\label{estuglobal}\begin{split}
E_{\e_n,\delta_n}(\un) &\leq E_{\e_n,\delta_n}(W_n)
\leq \sum_{i=1}^N \undemi \mu(B_n^i,\lambda\e_n,R,\kappa_i) +C_2
\end{split}\end{equation}
with $C_2$ not depending on $n$.

\medskip On the other hand, as the restriction of $\displaystyle \frac{\un}{|\un|}
$ to each annulus $\Gamma_n^i(R)$, translated by $a_i$, lies in $\mathcal V_{\kappa_i}(\lambda\e_n,R)$, the mere definition of $\mu(B_n^i,\lambda\e_n,R,\kappa_i)$ in Notations \ref{defdegkappa} implies, for $n\geq 1$ and $i=1,\cdots,N$
\begin{equation}
\int_{\Gamma_n^i(R)} \nabla \frac{\un}{|\un|}\cdot A_\delta(x) \nabla \frac{\un}{|\un|}\geq \mu(B_n^i,\lambda\e_n,R,\kappa_i)\,.
\end{equation}
Proposition \ref{usuru} provides then a constant $C$ such that
\begin{equation}\label{estunann}
\int_{\cup_i \Gamma_n^i(R)}\nabla \un\cdot A_\delta(x)\nabla \un \geq \sum_{i=1}^N \mu(B_n^i,\lambda\e_n,R,\kappa_i)-NC\,.
\end{equation}

Comparing (\ref{estuglobal}) and (\ref{estunann}) provides at last
\begin{equation}\begin{split}
\frac{m}{2} \int_{\Omega\backslash (\,\cup_i \Gamma_n^i(R)\,)} & |\nabla \un|^2  \leq \undemi \int_{\Omega\backslash (\,\cup_i \Gamma_n^i(R)\,)}\nabla u_n\cdot A_\delta \nabla u_n \\
&\leq E_{\e_n,\delta_n}(\un)\,-\, \undemi \int_{\cup_i \Gamma_n^i(R)}\nabla \un\cdot A_\delta(x)\nabla \un \\
&\leq \undemi(C_2+NC)
\end{split}\end{equation}
which contradicts (\ref{contrad}). The theorem is proved.

\hfill $\square$

\vskip1.5cm
\section {The homogenized equation for $u_\infty$.}

\bigskip \subsection {\label{71} The mean homogenized matrix $A^0$ (cf. [SP], [Ba], [BLP])}\-

$Y$ is the cell $[0,1[\times [0,1[$ in $\R^2$.

We start with the obvious following remark, which shall be of constant use\,:
\begin{equation}\label{TAinv}\begin{split}
\forall\,f\in L^2(Y)\,,\; \exists\,!\, g\in H^1_{per}(\overline Y)\,\text{ s.t. } \\
div\, A(y)\nabla g(y)=f(y) \,\text{ and }\, \displaystyle \int_Yg(y)dy=0\,.
\end{split}\end{equation}

Accordingly, one defines the vector field $\widehat \chi(y))=\big(\widehat \chi^j(y)\big)_{j=1,2}$ on $Y$ as the (unique) solution of the system of equations
\begin{equation}\label{eqchi}
div\,A(y)\nabla\widehat \chi^j(y)=\sum_i \frac{\partial A_{i,j}(y)}{\partial y_i}\,,\; \widehat \chi^j \in H^1_{per}(\overline Y)\,,\;\int_Y \widehat \chi^j(y)=0\,.
\end{equation}

The mean homogenized matrix $A^0$ is a matrix with constant entries, defined by
\begin{equation}
A^0_{ij} =\int_Y A_{ij}(y)dy -\int_Y \sum_k A_{ik}(y)\frac{\partial \widehat \chi^j(y)}{\partial y_k}dy\,.
\end{equation}

\subsection{Statement of the result} The general theorem about nonlinear Ginzburg-Landau type equations can be stated as follows\,:
\smallskip

\begin{thm}
\label{thmhomo} Fix $\Omega_0$ a bounded domain in $\R^2$.

Let $(\delta_n)$ be a sequence tending to $0$ and  $(u_n)$ a sequence in $ H^1(\overline \Omega_0,\C)$ satisfying the following assumptions\,:
\begin{itemize}
\item [i.] $|u_n|\leq 1$ and $\displaystyle \lim_{n\to \infty} \int_{\Omega_0}\big(\,1-|u_n| \,\big)^2=0$\,.
\item[ii.] $\sup_n ||\nabla u_n||_2 <+\infty$\,.
\item[iii.] $\displaystyle -div\,\Big(A\big(\frac{x}{\delta_n}\big)\nabla u_n\Big)(x) =u_n(x)\,f_n(x,u_n)$ for some real valued function $f_n$ on $\Omega_0\times \C$, depending on $n$.
\end{itemize}

Then, any weak $H^1$-limit $u_\infty$ of the $u_n$ is a $A^0$-harmonic function in $H^1(\Omega_0,S^1)$, i.e. a weak solution of the equation
\begin{equation}\label{equn}
-div\,(\,A^0\,\nabla u_\infty \,\big)= u_\infty\,\big( \nabla u_\infty\cdot A^0\, \nabla u_\infty\big)\,.
\end{equation}
\end{thm}

\medskip
\subsection{Proof of Theorem \ref{thmhomo}}\-

The fact that the limit equation for $u_\infty$ should be driven by the matrix $A^0$ is quite expected, but we can provide here a very quick and simple proof, based on the periodic unfolding method (cf. [CDG]).

 The unfolding operator $T_\delta$ ($\delta>0$) is described as follows\,: for $f\in L^2(\Omega_0)$, $T_\delta f$ is the function on $\Omega_0\times Y$ defined by
 \begin{equation}
T_\delta f(x,y)=\left\{ \begin{matrix} f\left(\delta \left[ \displaystyle\frac{x}{\delta}\right] +\delta y\right) &\text{ if } \delta \left[\displaystyle \frac{x}{\delta}\right] +\delta Y\subset \Omega_0\,, \\ \\0 &\text{ otherwise}\,,
\end{matrix} \right.
\end{equation}
where $\left[\displaystyle \frac{x}{\delta}\right]\in Z^2 $ is the integer part of $\displaystyle \frac{x}{\delta}$, i.e. the only $z\in Z^2$ such that $\displaystyle \frac{x}{\delta}-z\in Y$.

One has $T_\delta (\nabla f)(x,y)=\delta \nabla_y T_\delta f(x,y)$ and
\begin{equation}\label{eclatTA}
T_\delta\big(div\,A(\frac{x}{\delta})\nabla f\big)(x,y)= \delta \,div_y\,\big(A(y) T_\delta(\nabla f) \big)(x,y) \,.
\end{equation}

\medskip Let us write equation (\ref{equn}) as
\begin{equation}
-div\,A\big(\frac{x}{\delta_n}\big)\nabla u_n\wedge u_n=0\,,
\end{equation}
and apply the operator $T_{\delta_n}$ in order to get
\begin{equation}\label{eqwedge}
-div_y A(y) T_{\delta_n}(\nabla u_n)(x,y) \wedge T_{\delta_n}(u_n)(x,y)=0\,.
\end{equation}
When $n\to \infty$, then, by the results of [CDG, Prop. 2.9 and Thm. 3.5], $T_{\delta_n}(u_n)(x,y)\to u_\infty(x)$ strongly in $L^2(\Omega\times Y)$, while $T_{\delta_n}(\nabla u_n)(x,y)\to \nabla u_\infty (x)+\nabla_y\widehat u(x,y)$ $L^2$-weakly, for some $\widehat u(x,y)\in L^2(\Omega,H^1_{per}(\overline Y))$ with vanishing mean \,: $\displaystyle \int_Y \widehat u(x,y)dy=0$, $x$ a.s..

So, passing to the weak limit, (\ref{eqwedge}) provides, in $L^2(\Omega_0,H^{-1}(\overline Y))$,
\begin{equation}
-div_y A(y)\big(\nabla_x u_\infty(x)+\nabla_y\widehat u(x,y)\big)\wedge u_\infty(x)=0\,
\end{equation}
or equivalently
\begin{equation}\label{equhat0}
div_y\Big(A(y)\nabla_y \big(\widehat u(x,y)\wedge u_\infty(x)\big)\Big) =-\sum_{ij} \frac{\partial A_{ij}(y)}{\partial y_i} \frac{\partial u_\infty (x)}{\partial x_j}\wedge u_\infty(x)\,.
\end{equation}
Invoking, for fixed $x$, linearity and uniqueness in (\ref{TAinv}), we get from (\ref{eqchi}) compared with (\ref{equhat0}), since $\widehat u$ is $Y$-periodic with vanishing mean\,:
\begin{equation}\label{equhat}
\widehat u(x,y)\wedge u_\infty(x)=-\widehat \chi(y)\cdot \nabla u_\infty(x)\wedge u_\infty(x)\,.
\end{equation}

\medskip Now, it suffices to pair equation (\ref{eqwedge}) with any $T_{\delta_n}(f)$, $f\in C^\infty_c(\Omega_0)$ (i.e. of class $C^\infty$ with compact support) to get, firstly when $\delta_n$ is small enough (i.e. such that $\delta_n \left[\displaystyle \frac{x}{\delta_n}\right] +\delta_n Y\subset \Omega_0$ for any $x\in\,supp\,f$
), then passing to the limit
\begin{equation}\begin{split}
0&=  \int_{\Omega_0\times Y}-div_y \,\Big(A(y) T_{\delta_n}(\nabla u_n)(x,y) \wedge T_{\delta_n}(u_n) (x,y)\Big)\, \frac{1}{\delta_n}T_{\delta_n} f(x,y)dxdy \\
 &=\int_{\Omega_0\times Y} A(y) T_{\delta_n}(\nabla u_n)(x,y) \wedge T_{\delta_n}(u_n)(x,y)\,\cdot\, \frac{1}{\delta_n}\nabla_yT_{\delta_n}  f(x,y)dxdy \\
 &=\int_{\Omega_0\times Y} A(y) T_{\delta_n}(\nabla u_n)(x,y) \wedge T_{\delta_n}(u_n)(x,y)\,\cdot\, T_{\delta_n} (\nabla_x f)(x,y))dxdy \\
 &=\int_{\Omega_0\times Y}A(y)\big(\nabla_xu_\infty(x)+\nabla_y \widehat u(x,y)\big)\wedge u_\infty(x)\, \cdot\,\nabla_xf(x)dxdy\,,
\end{split}\end{equation}
i.e.
\begin{equation}
-div_x \int_YA(y) \big(\nabla _xu_\infty(x)+\nabla_y\widehat u(x,y)\big)\wedge u_\infty(x)dy=0\,,\; x \text{ a.s.}
\end{equation}
Introducing (\ref{equhat}), we have then
\begin{equation}
0=-div\, A^0 \Big(\nabla u_\infty(x)\wedge u_\infty(x)\Big) \end{equation}
which can be written
\begin{equation}\label{uinfini1}
0=  -\Big(div\,A^0 \nabla u_\infty(x)\Big)\wedge u_\infty(x)
\end{equation}
in $H^{-1}(\Omega_0)$, with $A^0$ defined in subsection \ref{71}.

\smallskip Finally, we notice that (\ref{uinfini1}) means
\begin{equation}
-div\,\big(A^0 \nabla u_\infty\big)(x)=u_\infty(x)\,f(x)
\end{equation}
for some real valued distribution $f$. As $|u_\infty|=1$ (by Assumption i.), $f$ is given by
\begin{equation}
f(x)=\Big(-div\,A^0\nabla u_\infty(x)\Big)\cdot u_\infty(x)= \nabla u_\infty(x)\cdot A^0 \nabla u_\infty(x)\,,
\end{equation}
and the theorem is proved.

\hfill $\square$

\end{document}